\documentclass{amsart}
\usepackage{amsmath}
\usepackage{amsfonts}
\usepackage{amssymb}
\usepackage{amsthm}
\usepackage[T1]{fontenc}
\usepackage{float}
\usepackage[all]{xy}
\newtheorem{theorem}{Theorem}[section]
\newtheorem{maprop}[theorem]{Proposition}
\newtheorem{monlem}[theorem]{Lemma}
\newtheorem{corol}[theorem]{Corollary}
\newtheorem{maconj}[theorem]{Conjecture}

\theoremstyle{definition}
\newtheorem{defi}[theorem]{Definition}
\newtheorem{notation}[theorem]{Notation}
\newtheorem{monexmp}[theorem]{Example}
\newtheorem{rmq}[theorem]{Remark}

\def\Aaffine{\tilde{\mathbb A}}
\def\Daffine{\tilde{\mathbb D}}
\def\Eaffine{\tilde{\mathbb E}}
\def\k{{\mathbb C}}

\def\ker{{\rm{Ker}\,}}
\def\coker{{\rm{Coker}\,}}
\def\Hom{{\rm{Hom}}}
\def\Ext{{\rm{Ext}}}
\def\modg{{\textrm{-}\rm{mod}}}
\def\Gr{{\rm{Gr}}}

\def\End{{\rm{End}}}
\def\dim{{\rm{dim}\,}}
\def\ddim{{\textbf{dim}\,}}

\def\rep{{\rm{rep}}}
\def\Ob{{\rm{Ob}}}

\def\den{{\rm{den}}}

\def\im{{\rm{im}\,}}
\def\re{{\rm{re}\,}}
\def\Sc{{\rm{Sc}\,}}

\def\A{{\mathbb{A}}}

\def\N{{\mathbb{N}}}
\def\Z{{\mathbb{Z}}}
\def\C{{\mathbb{C}}}

\def\P{{\mathbb{P}}}

\def\CC{{\mathcal{C}}}

\def\<{\left<}
\def\>{\right>}

\def\d{{\partial}}

\def\ens#1{\left\{ #1 \right\}}
\def\fl{{\longrightarrow}\,}

\title{Generic variables in acyclic cluster algebras}
\author{G. Dupont}
\address{
	Universit\'e de Sherbrooke\\
	2500, Boul. de l'Universit\'e\\
	J1K 2R1, Sherbrooke, QC, Canada.
}
\email{gregoire.dupont@usherbrooke.ca}
\urladdr{
	http://pages.usherbrooke.ca/gdupont2
}

\begin{document}

\begin{abstract}
	Let $Q$ be an acyclic quiver. We introduce the notion of generic variables for the coefficient-free acyclic cluster algebra $\mathcal A(Q)$. We prove that the set $\mathcal G(Q)$ of generic variables contains naturally the set $\mathcal M(Q)$ of cluster monomials in $\mathcal A(Q)$ and that these two sets coincide if and only if $Q$ is a Dynkin quiver. We establish multiplicative properties of these generic variables analogous to multiplicative properties of Lusztig's dual semicanonical basis. This allows to compute explicitly the generic variables when $Q$ is a quiver of affine type. When $Q$ is the Kronecker quiver, the set $\mathcal G(Q)$ is a $\mathbb Z$-basis of $\mathcal A(Q)$ and this basis is compared to Sherman-Zelevinsky and Caldero-Zelevinsky bases.
\end{abstract}

\maketitle

\setcounter{tocdepth}{1}
\tableofcontents

\section{Introduction}
	Cluster algebras were introduced by Fomin and Zelevinsky in order to provide a combinatorial framework for studying total positivity in algebraic groups and canonical bases in quantum groups \cite{cluster1}. Since then, they were subjects to important developments in various areas of mathematics like combinatorics, Lie theory, Poisson geometry, Teichmüller theory and representation theory of quivers. Nevertheless, the problem of computing bases in arbitrary cluster algebras is still widely open. Sherman-Zelevinsky, Caldero-Zelevinsky and Cerulli provided some explicit constructions in particular cases \cite{shermanz,CZ,Cerulli:A21}. More generally, Geiss, Leclerc and Schr\"oer gave a construction of bases in a wide class of cluster algebras, including acyclic cluster algebras \cite{GLS:KMgroups}. Their construction, using representation theory of preprojective algebras, is rather theoretical and difficult to explicit in practice. Inspired by their works, we provide in this article a similar, but easier to explicit, construction using cluster categories and cluster characters. 

	Let $Q$ be a quiver, we denote by $Q_0$ its set of vertices and by $Q_1$ its set of arrows. We always assume that $Q_0$ and $Q_1$ are finite sets and that the underlying unoriented graph of $Q$ is connected. We fix an \emph{acyclic} quiver $Q$ (that is, without oriented cycles) and a family $\textbf u =(u_i|i\in Q_0)$ of indeterminates over $\Z$. We denote by $\mathcal A(Q,\textbf u)$, or simply $\mathcal A(Q)$, the (coefficient-free) cluster algebra with initial seed $(Q,\textbf u)$. According to the \emph{Laurent phenomenon}, it is a $\Z$-subalgebra of the algebra $\Z[\textbf u^{\pm 1}]$ of Laurent polynomials in the $u_i$'s \cite{cluster1}. The cluster algebra $\mathcal A(Q)$ is endowed with a distinguished set of generators, called \emph{cluster variables}, gathered into possibly overlapping sets of fixed cardinality called \emph{clusters}. Monomials in variables belonging to a same cluster are called \emph{cluster monomials}. We denote by $\mathcal M(Q)$ the set of all cluster monomials in $\mathcal A(Q)$. For every element $x$ in the ring $\Z[\textbf u^{\pm 1}]$, there exists an unique $Q_0$-tuple $\den(x)=(d_i)_{i \in Q_0} \in \Z^{Q_0}$, called \emph{denominator vector of $x$} such that 
	$$x=\frac{P(u_i|i\in Q_0)}{\prod_{i \in Q_0}u_i^{d_i}}$$
	where $P(u_i|i\in Q_0)$ is a polynomial not divisible by any $u_i$.

	Let $\k$ denote the field of complex numbers. Throughout the article, $H$ will denote the finite dimensional hereditary path algebra $\k Q$ of $Q$. We denote by $H$-mod the category of finitely generated left-$H$-modules. As usual, this category is identified with the category $\rep(Q)$ of finite dimensional representations of $Q$ over $\C$. We denote by $D^b(H\modg)$ the bounded derived category of $H$-mod with shift functor $[1]$ and Auslander-Reiten translation $\tau$. The \emph{cluster category} $\CC_Q$ of $Q$ was introduced in \cite{BMRRT} in order to categorify the cluster algebra $\mathcal A(Q)$ (see also \cite{CCS1} for Dynkin type $\A$). It is defined as the orbit category of the functor $\tau^{-1}[1]$ in $D^b(H\modg)$. It is a Krull-Schmidt triangulated category satisfying the 2-Calabi-Yau property, that is, 
	$$\Ext^1_{\CC_Q}(M,N) \simeq D \Ext^1_{\CC_Q}(N,M)$$
	for any two objects $M,N \in \CC_Q$ where $D=\Hom_{\k}(-,\k)$ \cite{K,BMRRT}. Moreover, the set of isoclasses of indecomposable objects in $\CC_Q$ can be identified with the disjoint union of the set of isoclasses of indecomposable $H$-modules and the set of isoclasses of shifts of indecomposable projective $H$-modules \cite{BMRRT}. If $M,N$ are $H$-modules, it is proved in \cite{BMRRT} that there is a bifunctorial isomorphism $\Ext^1_{\CC_Q}(M,N) \simeq \Ext^1_{H}(M,N) \oplus D\Ext^1_{H}(N,M)$. In particular, a $H$-module is rigid (that is, without self-extensions) in $H$-mod if and only if it is rigid in $\CC_Q$.

	The cluster category $\CC_Q$ is known to provide a fruitful framework for studying the cluster algebra $\mathcal A(Q)$ (see e.g. \cite{BMRRT,BMRT,CC,CK1,CK2}). In \cite{CC,CK2}, the authors introduced a map $X_?$ from the set $\Ob(\CC_Q)$ to the ring $\Z[\textbf u^{\pm 1}]$, called \emph{Caldero-Chapoton map}, as a certain generating series of Euler characteristics of grassmannians of submodules (see Section \ref{section:generic} for definitions). In \cite[Theorem 4]{CK2}, the authors proved that $X_?$ induces a 1-1 correspondence between the set of isoclasses of indecomposable rigid objects in $\CC_Q$ and the set of cluster variables in $\mathcal A(Q)$. Moreover, they proved that the set of clusters in $\mathcal A(Q)$ coincides with the set of sets $\ens{X_{T_i}|i \in Q_0}$ where $T=\bigoplus_{i\in Q_0}T_i$ runs over the isoclasses of cluster-tilting objects in $\CC_Q$. In particular, it follows that 
	$$\mathcal M(Q)=\ens{X_M|M \textrm{ is rigid in }\CC_Q}.$$
	
	When $Q$ is a quiver of Dynkin type, Caldero and Keller proved that $\mathcal M(Q)$ is a $\Z$-basis in the cluster algebra $\mathcal A(Q)$ \cite[Corollary 3]{CK1}. When $Q$ is not of Dynkin type, it was observed by Sherman and Zelevinsky that $\mathcal M(Q)$ does not span the cluster algebra $\mathcal A(Q)$. Nevertheless, it is known that $\mathcal M(Q)$ is still linearly independent over $\Z$ (see e.g. \cite{GLS:KMgroups}). Thus, it seems natural to try to ``complete'' $\mathcal M(Q)$ into a $\Z$-basis of $\mathcal A(Q)$.

	In this paper, we introduce a new set $\mathcal G(Q)$ of Laurent polynomials, called \emph{generic variables}, in the ring $\Z[\textbf u^{\pm 1}]$ of Laurent polynomials containing the cluster algebra $\mathcal A(Q)$. We prove that the denominator vector map induces a parametrization of the set $\mathcal G(Q)$ by the root lattice $\Z^{Q_0}$. We establish in Proposition \ref{prop:clustermonomials} that $\mathcal G(Q)$ naturally contains the set $\mathcal M(Q)$ of cluster monomials. Moreover, we prove that these two sets coincide if and only if $Q$ is a Dynkin quiver. We conjecture in Section \ref{section:conjectures} that these generic variables form a $\Z$-basis in the acyclic cluster algebra $\mathcal A(Q)$. 

	In Proposition \ref{prop:Kacdcp} and Lemma \ref{lem:multiplicativity}, we establish multiplicative properties for these cluster variables, notably with respect to Kac's canonical decomposition. These multiplicative properties allow to simplify considerably the problem of the computation of generic variables. In particular, when $Q$ is a quiver of affine type, Proposition \ref{prop:explicitbase} provides an explicit description of the set $\mathcal G(Q)$. We also prove that for an affine quiver $Q$, each character $X_M$ belong to the cluster algebra $\mathcal A(Q)$ so that $\mathcal A(Q)=\Z[X_M|M \in \Ob(\CC_Q)]$. It follows that, in this case, the set of generic variables actually lies in the cluster algebra $\mathcal A(Q)$.

	Finally, when $Q$ is the Kronecker quiver, we compare our construction with two already known constructions of bases in $\mathcal A(Q)$ due respectively to Sherman-Zelevinsky \cite{shermanz} and Caldero-Zelevinsky \cite{CZ}. In particular, it follows that, in this case, generic variables form a $\Z$-basis of the cluster algebra $\mathcal A(Q)$.

\begin{section}{Generic Variables}\label{section:generic}
	Before recalling the definition of the Caldero-Chapoton map, we need some terminology and notations. For any vertex $i \in Q_0$, we denote by $S_i$ the associated simple $H$-module and by $P_i$ its projective cover. The \emph{dimension vector} of a $H$-module $M$ is the element
	$$\ddim M=(\dim \Hom_{H}(P_i,M))_{i \in Q_0} \in \N^{Q_0}.$$
	We denote by $K_0(H\modg)$ the Grothendieck group of the category $H$-mod. The dimension vector map induces an isomorphism $K_0(H\modg) \simeq \Z^{Q_0}$ sending the isoclass of the simple $H$-module $S_i$ to the $i$-th vector $\alpha_i$ of the canonical basis of $\Z^{Q_0}$ for any $i \in Q_0$. 
	
	We denote by $\<-,-\>$ the homological Euler form on $H$-mod. It is defined by
	$$\<M,N\>=\dim \Hom_{H}(M,N)-\dim \Ext^1_{H}(M,N)$$
	for any two $H$-modules $M,N$. This induces a well-defined bilinear form on the Grothendieck group $K_0(H\modg)$.

	For any $H$-module $M$ and any $\textbf e \in \N^{Q_0}$, the \emph{grassmannian of submodules of $M$ of dimension $\textbf e$} is the set
	$$\Gr_{\textbf e}(M)=\ens{N \textrm{ submodule of }M | \ddim N = \textbf e}.$$
	This is a closed subset of the ordinary vector spaces grassmannian so that it is a projective variety whose Euler-Poincar\'e characteristic with respect to the simplicial cohomology is denoted by $\chi(\Gr_{\textbf e}(M))$. 

	\begin{defi}\cite{CC}
		Let $Q$ be an acyclic quiver. The \emph{Caldero-Chapoton map} is the map 
		$$X_? : \Ob(\CC_Q) \fl \Z[\textbf u^{\pm 1}]$$
		given by~:
		\begin{enumerate}
			\item[a.] If $M$ is an indecomposable $H$-module, 
				\begin{equation}\label{eq:XM}
					X_M=\sum_{\textbf e \in \N^{Q_0}} \chi(\Gr_{\textbf e}(M)) \prod_{i \in Q_0} u_i^{-\<\textbf e,\alpha_i\>-\<\alpha_i, \ddim M-\textbf e\>};
				\end{equation}
			\item[b.] if $M \simeq P_i[1]$ is the shift of an indecomposable projective module, 
				$$X_M=u_i;$$
			\item[c.] for any two objects $M,N$ in $\mathcal C_Q$, 
				$$X_MX_N=X_{M \oplus N}.$$
		\end{enumerate}
	\end{defi}
	It is easy to check that equation (\ref{eq:XM}) also holds for decomposable $H$-modules (see e.g. \cite{CC}). Note also that the Caldero-Chapoton map is invariant under isomorphisms.

	For any arrow $\alpha: i \fl j \in Q_1$, we denote by $s(\alpha)=i$ its \emph{source} and by $t(\alpha)=j$ its \emph{target}. For any dimension vector $\textbf d \in \N^{Q_0}$, we denote by $\rep(Q,\textbf d)$ the \emph{representation variety of dimension $\textbf d$}, that is, the set of representations of $Q$ with dimension vector $\textbf d$. It can be identified with the irreducible affine variety $\prod_{\alpha \in Q_1} k^{d_{s(\alpha)}} \times k^{d_{t(\alpha)}}$.
	
	For any dimension vector $\textbf d \in \N^{Q_0}$, the algebraic group 
	$$GL(\textbf d)=\prod_{i \in Q_0}GL(d_i,\C)$$
	acts algebraically on the variety $\rep(Q,\textbf d)$ by conjugation. The orbits of this action are in bijection with the isoclasses of representations of $Q$ of dimension vector $\textbf d$. In particular, two points $M,N$ in $\rep(Q,\textbf d)$ are in the same orbit if and only if they represent isomorphic representations.

	\begin{monlem}\label{lem:Ude}
		Fix $Q$ an acyclic quiver and $\textbf d \in \N^{Q_0}$. For every $\textbf e \in \N^{Q_0}$, there exists a $GL(\textbf d)$-invariant dense open subset $U_{\textbf d,\textbf e} \subset \rep(Q,\textbf d)$ such that the map 
		$$\phi_{\textbf d,\textbf e}:M \mapsto \chi(\Gr_{\textbf e}(M))$$ 
		is constant on $U_{\textbf d, \textbf e}$. Moreover, the value of $\phi_{\textbf d,\textbf e}$ on such a subset does not depend on the choice of $U_{\textbf d, \textbf e}$.
	\end{monlem}
	\begin{proof}
		It follows from \cite[Proposition 1.3]{Xu} that $\phi_{\textbf d,\textbf e}$ is a $GL(\textbf d)$-invariant constructible function on $\rep(Q,\textbf d)$. Since $\rep(Q,\textbf d)$ is an irreducible variety there exists a dense open subset $U_{\textbf d,\textbf e}$ in $\rep(Q,\textbf d)$ on which this map is constant. By irreducibility, any two non-empty sets intersect so that the value of $\phi_{\textbf d,\textbf e}$ does not depend on the choice of $U_{\textbf d,\textbf e}$. Moreover, since $\phi_{\textbf d,\textbf e}$ is $GL(\textbf d)$-invariant, we can assume that $U_{\textbf d,\textbf e}$ is $GL(\textbf d)$-invariant.
	\end{proof}

	\begin{notation}
		Fix $\textbf d, \textbf e \in \N^{Q_0}$, we use the following notations~:
		$$\textbf e \leq \textbf d \Leftrightarrow e_i \leq d_i \textrm{ for every }i \in Q_0~;$$
		$$\textbf e \lneqq \textbf d \Leftrightarrow \textbf e \leq \textbf d \textrm{ and } \textbf e \neq \textbf d.$$
	\end{notation}

	\begin{corol}\label{corol:Ud}
		Fix $Q$ an acyclic quiver and $\textbf d \in \N^{Q_0}$. For any $\textbf e \in \N^{Q_0}$, let $U_{\textbf d,\textbf e} \subset \rep(Q,\textbf d)$ be an open dense subset provided by Lemma \ref{lem:Ude} and let $U_{\textbf d}=\bigcap_{\textbf e \leq \textbf d} U_{\textbf d, \textbf e}$. Then, $U_{\textbf d}$ is a $GL(\textbf d)$-invariant open subset of $\rep(Q,\textbf d)$ on which $X_?$ is constant. Moreover, if $U'_{\textbf d}$ is another non-empty open subset of $\rep(Q,\textbf d)$ on which $X_?$ is constant, then the values of $X_?$ on $U_{\textbf d}$ and $U'_{\textbf d}$ coincide.
	\end{corol}
	\begin{proof}
		$U_{\textbf d}=\bigcap_{\textbf e \leq \textbf d} U_{\textbf d, \textbf e}$ is a finite intersection of $GL(\textbf d)$-invariant dense open subsets, it is thus a $GL(\textbf d)$-invariant open dense subset in $\rep(Q,\textbf d)$. If $\textbf e \leq \textbf d$, we denote by $n_{\textbf e}$ the value of $\phi_{\textbf d,\textbf e}$ on $U_{\textbf d, \textbf e}$. For any $M \in U_{\textbf d}$, we thus have
		\begin{align*}
			X_M 
				&= \sum_{\textbf e \leq \textbf d} \chi(\Gr_{\textbf e}(M)) \prod_i u_i^{-\<\textbf e, \alpha_i\>-\<\alpha_i, \textbf d - \textbf e\>}\\
				&=\sum_{\textbf e \leq \textbf d} n_{\textbf e} \prod_i u_i^{-\<\textbf e, \alpha_i\>-\<\alpha_i, \textbf d - \textbf e\>}\\
		\end{align*}
		so that $X_?$ is constant over $U_{\textbf d}$. The second assertion follows from the irreducibility of $\rep(Q,\textbf d)$.
	\end{proof}
	
	Fix an object $M$ in $\CC_Q$, then $M$ admits a unique (up to isomorphism) decomposition $M=H^0(M) \oplus P_M[1]$ where $H^0(M)$ is a $H$-module and $P_M$ is a projective $H$-module. 

	The dimension vector map on the cluster category is the additive map $\textbf{dim}_{\CC_Q} : \Ob(\CC_Q) \fl \Z^{Q_0}$ defined by $\ddim_{\CC_Q}(P_i[1])=-\alpha_i$ for any $i \in Q_0$ and $\ddim_{\CC_Q}(M)=\ddim M$ for any indecomposable $H$-module $M$. Since $\textbf{dim}_{\CC_Q}(M)=\ddim M$ for any $H$-module $M$, we abuse the notations and simply write $\ddim M$ for an arbitrary object in the cluster category.

	\begin{notation}
		If $\textbf d \in \Z^{Q_0}$, we set
		$$[\textbf d]_+=\left(\max(d_i,0)\right)_{i \in Q_0}, \, [\textbf d]_-=\left(\min(d_i,0)\right)_{i \in Q_0}$$
		and
		$$P_{\textbf d}[1]=\bigoplus_{d_i<0} P_i[1]^{\oplus (-d_i)}.$$
	\end{notation}
	
	\begin{rmq}
		If $M$ is a rigid object in $\CC_Q$, then
		\begin{equation}\label{eq:rigiddimension}
			\ddim H^0(M)= [\ddim M]_+ \textrm{ and } \ddim P_M[1]=[\ddim M]_-.
		\end{equation}
		Indeed, if $P_i[1]$ is a direct summand of $M$ then $(\ddim H^0(M))_i = \dim \Hom_{H}(P_i,H^0(M)) \leq \dim \Hom_{\CC_Q}(P_i,H^0(M)) = \dim \Ext^1_{\CC_Q}(P_i[1],H^0(M)) = 0$. Thus, if $(\ddim H^0(M))_i>0$, we get $(\ddim M)_i=(\ddim H^0(M))_i$ and 
		$$[\ddim M]_+=\ddim H^0(M), \, [\ddim M]_-=\ddim P_M[1].$$ 
	
		Note nevertheless that equality (\ref{eq:rigiddimension}) does not hold in general if $M$ is not rigid. For example, consider the quiver $Q$ of Dynkin type $\mathbb A_1$, $S$ the unique simple (necessarily projective) $H$-module and $M=S \oplus S[1]$. Then $\ddim M=0$ whereas $H^0(M)=S$ and $P_M[1]=S[1]$ have respectively dimension vectors 1 and -1.
	\end{rmq}

	\begin{defi}
		Fix $\textbf d \in \N^{Q_0}$, we denote by $X_{\textbf d}$ the value of the Caldero-Chapoton map on $U_{\textbf d}$. Fix now $\textbf d \in \Z^{Q_0}$, we set
		$$X_{\textbf d} = X_{[\textbf d]_+}X_{P_{\textbf d}[1]}  = X_{[\textbf d]_+}.\prod_{d_i<0} u_i^{-d_i}$$
		and $X_{\textbf d}$ is called the \emph{generic variable of dimension $\textbf d$}.
		
		We denote by 
		$$\mathcal G(Q)=\ens{X_{\textbf d}|\textbf d \in \Z^{Q_0}}$$
		the set of all generic variables associated to the quiver $Q$.
	\end{defi}

	\begin{rmq}\label{rmq:denXd}
		According to \cite[Theorem 3]{CK2}, it is known that $\den(X_M)=\ddim M$ for any object $M$ in the cluster category $\CC_Q$. In particular, for every $\textbf d \in \Z^{Q_0}$, we have $\den(X_{\textbf d})=\textbf d$ so that generic variables are naturally parametrized by $\Z^{Q_0}$.
	\end{rmq}

	\begin{rmq}
		We set $\im(X_?)=\Z[X_M|M \in \Ob(\CC_Q)]$. It follows from \cite[Theorem 4]{CK2} that 
		$$\mathcal A(Q)=\Z[X_M|M \textrm{ is rigid in}\CC_Q] \subset \im(X_?).$$
		Nevertheless, it is not known in general if the inclusion is proper or not. Thus, one should take care that it is not clear from the definition that $\mathcal G(Q) \subset \mathcal A(Q)$. We will prove in Theorem \ref{theorem:imCC} that, when $Q$ is a quiver of affine type, $\im(X_?)=\mathcal A(Q)$ so that $\mathcal G(Q) \subset \mathcal A(Q)$.
	\end{rmq}
\end{section}

\begin{section}{Multiplicative properties}
	In this section, we study interactions between generic variables and classical results of Kac concerning generic representations of quivers. We fix an acyclic quiver $Q$ and we denote by $\mathfrak g_Q$ the corresponding affine Kac-Moody-Lie algebra. We only need basic results on root systems of affine Kac-Moody-Lie algebras. For the readers unfamiliar with this theory, we provide a simple combinatorial description of the root system of $Q$.

	We denote by $(-,-)$ the \emph{Tits form} of $Q$, that is, the $\Z$-bilinear form on $\Z^{Q_0}$ given by 
	$$(\textbf e,\textbf f)=\sum_{i \in Q_0} e_if_i - \sum_{\alpha \in Q_1}e_{s(\alpha)}f_{t(\alpha)}$$
	for any $\textbf e, \textbf f \in \Z^{Q_0}$. It is well-known that $(\textbf e,\textbf e)=\<\textbf e,\textbf e\>$ for any $\textbf e \in \Z^{Q_0}$. 
	We denote by
	$$\Phi_{>0}(Q)=\ens{\textbf e \in \N^{Q_0}|(\textbf e,\textbf e) \leq 1}$$
	the set of \emph{positive roots} of $Q$. A positive root is called \emph{real} if $(\textbf e,\textbf e) = 1$ and is called \emph{imaginary} otherwise. We denote by $\Phi_{>0}^{\re}(Q)$ the set of positive real roots of $Q$ and by $\Phi^{\im}_{>0}(Q)$ the set of positive imaginary roots of $Q$. The \emph{root lattice} of $Q$ is $\Z^{Q_0}$ and the elements $\alpha_i, i \in Q_0$ are called the \emph{simple roots} of $Q$.
	
	According to Kac's theorem, if $\textbf d \in \N^{Q_0}$, there exists an indecomposable representation in $\rep(Q,\textbf d)$ if and only if $\textbf d$ is a positive root of $Q$. Moreover, this representation is unique (up to isomorphism) if and only if $\textbf d$ is a positive real root. We say that a (necessarily positive) root is a \emph{Schur root} if there exists a (necessarily indecomposable) representation $M \in \rep(Q,\textbf d)$ such that $\End_{H}(M) \simeq \C$. Such a representation is called a \emph{Schur} representation. The set of Schur roots is denoted by $\Phi^{\Sc}(Q)$.

	Let $\textbf d \in \N^{Q_0}$. According to \cite{Kac:infroot1}, there exists a dense open subset $\mathfrak M_{\textbf d} \subset \rep(Q,\textbf d)$ and a family $\ens{\textbf e_1, \ldots, \textbf e_n}$ of Schur roots such that every representation $M \in \mathfrak M_{\textbf d}$ decomposes into  $M=\bigoplus_{i=1}^n M_i$ where each $M_i$ is an indecomposable Schur representation of dimension $\textbf e_i$. In particular $\textbf d=\sum_{i=1}^n \textbf e_i$ and the family $\ens{\textbf e_1, \ldots, \textbf e_n}$ is uniquely determined. This decomposition of $\textbf d$ is called the \emph{canonical decomposition of $\textbf d$} and is denoted by $\textbf d = \textbf e_1 \oplus \cdots \oplus \textbf e_n$.
	
	The canonical decomposition can be characterized as follows~:
	\begin{maprop}[\cite{Kac:infroot2}]\label{prop:Kacdcp}
		Fix $Q$ an acyclic quiver and $\textbf d \in \N^{Q_0}$. Then $\textbf d=\bigoplus_{i=1}^n \textbf e_i$ if and only if for every $i \in \ens{1, \ldots, n}$, there exists a Schur representation $M_i \in \rep(Q,\textbf e_i)$ such that $\Ext^1_{kQ}(M_k,M_l)=0$ for any $k \neq l$ in $\ens{1, \ldots, n}$. 
	\end{maprop}

	Following \cite{Schofield:generalrepresentations}, we set~:
	\begin{defi}
		For any $\textbf d, \textbf d' \in \N^{Q_0}$, we say that $\Ext^1_{kQ}(\textbf d, \textbf d')$ \emph{vanishes generally} if there exists $M \in \rep(Q,\textbf d)$, $M' \in \rep(Q,\textbf d')$ such that 
		$$\Ext^1_{kQ}(M,M')=0$$
		in which case we write $\Ext^1_{kQ}(\textbf d, \textbf d')=0$.
	\end{defi}

	For any dimension vector $\textbf d \in \N^{Q_0}$, one has $U_{\textbf d} \cap \mathfrak M_{\textbf d} \neq \emptyset$, thus $X_{\textbf d}=X_M$ for some representation $M \in \mathfrak M_{\textbf d}$. One can thus use the canonical decomposition to compute generic variables as stated in the following proposition.
	
	\begin{maprop}\label{prop:dcpcanonique}
		Fix an acyclic quiver $Q$ and a dimension vector $\textbf d \in \N^{Q_0}$ with canonical decomposition $\textbf d= \textbf e_1 \oplus \cdots \oplus \textbf e_n$. Then
		$$X_{\textbf d}=\prod_{i=1}^n X_{\textbf e_i}.$$
	\end{maprop}
	\begin{proof}
		Consider the injective algebraic morphism :
		$$\phi : \left\{ \begin{array}{rcl}
			\rep(Q,\textbf e_1) \times \cdots \times \rep(Q,\textbf e_n) & \fl & \rep(Q,\textbf d)\\
			(M_1, \ldots, M_n) & \mapsto & \bigoplus_{i=1}^n M_i
		\end{array}\right.$$

		According to Proposition \ref{prop:Kacdcp}, as $\textbf d=\textbf e_1 \oplus \cdots \oplus \textbf e_n$ is the canonical decomposition of $\textbf d$, one has $\Ext^1_{kQ}(\textbf e_i, \textbf e_j)=0$ for every $i \neq j$. It thus follows from \cite{CBS} that $\phi$ is a dominant morphism.
		
		We set
		$$\mathcal U=(\mathfrak M_{\textbf e_1} \cap U_{\textbf e_1})\times \cdots \times (\mathfrak M_{\textbf e_n} \cap U_{\textbf e_n}).$$
		This is a dense open subset in $\rep(Q,\textbf e_1) \times \cdots \times \rep(Q,\textbf e_n)$ and thus $\phi(\mathcal U)$ is a dense open subset in $\rep(Q,\textbf d)$. In particular $\phi(\mathcal U) \cap U_{\textbf d} \neq \emptyset$ and we can choose some $M \in \phi(\mathcal U) \cap U_{\textbf d}$. We thus have $X_{\textbf d}=X_M$ and $M$ decomposes into a direct sum 
		$$M=\bigoplus_{i=1}^n M_i$$
		with $M_i \in \mathfrak M_{\textbf e_i} \cap U_{\textbf e_i}$. It follows that 
		$$X_\textbf d = X_M = X_{\bigoplus_{i=1}^n M_i} = \prod_{i=1}^n X_{M_i} = \prod_{i=1}^n X_{\textbf e_i}.$$
	\end{proof}
	
	We extend the notion of general vanishing of the Ext-spaces to the cluster category $\CC_Q$. This notion will be useful for proving a more general multiplicative property of generic variables.
	\begin{defi}
		Let $\textbf d, \textbf d'$ be two elements in $\Z^{Q_0}$, we say that $\Ext^1_{\mathcal C_Q}(\textbf d, \textbf d')$ \emph{vanishes generally} if there exists $M \in \rep(Q, [\textbf d]_+)$, $M' \in \rep(Q,[\textbf d']_+)$ such that
		$$\Ext^1_{\mathcal C_Q}\left( M \oplus P_{\textbf d}[1], M' \oplus P_{\textbf d'}[1]\right)=0$$
		in which case we write $\Ext^1_{\mathcal C_Q}(\textbf d, \textbf d')=0$.
	\end{defi}
	Note in particular that for $\textbf d, \textbf d' \in \N^{Q_0}$, if $\Ext^1_{\mathcal C_Q}(\textbf d, \textbf d')=0$ then, $\Ext^1_{kQ}(\textbf d, \textbf d')=0$ and $\Ext^1_{kQ}(\textbf d', \textbf d)=0$.

	We now prove an analogue of Geiss-Leclerc-Schr\"oer multiplicative property for the dual semicanonical basis \cite{GLS}~:
	\begin{monlem}\label{lem:multiplicativity}
		Let $Q$ be an acyclic quiver. Fix $\textbf d, \textbf d' \in \Z^{Q_0}$ and assume that $\Ext^1_{\mathcal C_Q}(\textbf d, \textbf d')=0$. Then, 
		$$X_{\textbf d}X_{\textbf d'}=X_{\textbf d+\textbf d'}.$$
	\end{monlem}
	\begin{proof}
		Assume that there is some $i \in Q_0$ such that $d_i<0$ and $d_i'>0$, then for every $M' \in \rep(Q,\textbf d')$, we have 
		$$0<d_i' =  \dim \Hom_{kQ}(P_i,M') \leq \dim \Ext^1_{\mathcal C_Q}(M',P_i[1])$$
		so that $\Ext^1_{\mathcal C_Q}(\textbf d, \textbf d')$ does not vanish generally. Thus, we can assume that $d_i$ and $d_i'$ are of the same sign for every $i \in Q_0$, in particular, $[\textbf d+\textbf d']_+=[\textbf d]_++[\textbf d']_+$.
		As $\Ext^1_{\mathcal C_Q}(\textbf d, \textbf d')=0$, we have $\Ext^1_{kQ}(\textbf d, \textbf d')=0$ and $\Ext^1_{kQ}(\textbf d', \textbf d)=0$. It thus follows from \cite{CBS} that the morphism
		$$\phi: \left\{\begin{array}{rcl}
			\rep(Q,[\textbf d]_+) \times \rep(Q,[\textbf d']_+) & \fl & \rep(Q,[\textbf d+\textbf d']_+)\\
			(U,V) & \mapsto & U\oplus V
		\end{array}\right.$$
		is dominant.
		As $U_{[\textbf d]_+}$ (respectively $U_{[\textbf d']_+}$) is open in $\rep(Q,[\textbf d]_+)$ (respectively in $\rep(Q,[\textbf d']_+)$), it follows that $U_{[\textbf d]_+} \oplus U_{[\textbf d']_+}$ is open and dense in $\rep(Q,[\textbf d+\textbf d']_+)$. Thus, $X_{[\textbf d+\textbf d']_+}=X_{[\textbf d]_+}X_{[\textbf d']_+}$ and it follows easily that $X_{\textbf d}X_{\textbf d'} = X_{\textbf d+\textbf d'}$.
	\end{proof}
\end{section}

\begin{section}{From cluster monomials to generic variables}
	We now prove that the set of generic variables always contains the set of cluster monomials of the cluster algebra $\mathcal A(Q)$ and that these sets coincide if and only if $Q$ is a Dynkin quiver.

	\begin{maprop}\label{prop:clustermonomials}
		Let $Q$ be an acyclic quiver. Then, the following hold~:
		\begin{enumerate}
			\item For any cluster monomial $c$ in $\mathcal A(Q)$, $c=X_{\den(c)}$~;
			\item For any rigid object $M$ in $\CC_Q$, $X_M=X_{\ddim M}$~; 
			\item $\mathcal M(Q) \subset \mathcal G(Q)$~; 
			\item $\mathcal M(Q) = \mathcal G(Q)$ if and only if $Q$ is a Dynkin quiver. 
		\end{enumerate}
	\end{maprop}
	\begin{proof}
		Let $c$ be a cluster monomial in $\mathcal A(Q)$, then there exists some rigid object $M$ in $\CC_Q$ such that $X_M=c$ and, according to \cite[Theorem 3]{CK2}, $\ddim M=\den(c)$. As $M$ is rigid, we have $[\ddim M]_+=\ddim H^0(M)$ and $[\ddim M]_-=\ddim P_M[1]$. The module $H^0(M)$ being rigid, its orbit is open and dense in $\rep(Q,[\ddim M]_+)$ so that $X_{H^0(M)}=X_{[\ddim M]_+}$. Since $M$ is a rigid object, we have $\ddim P_M[1]=[\ddim M]_-$ and thus $X_{P_M[1]}=X_{[\ddim M]_-}$. It follows that $c = X_M = X_{H^0(M)}X_{P_M[1]} = X_{[\ddim M]_+}X_{[\ddim M]_-}$ but $\Ext^1_{\mathcal C_Q}([\ddim M]_+,[\ddim M]_-)$ vanishes generally since $\Ext^1_{\CC_Q}(H^0(M),P_M[1])=0$. Thus, it follows from Lemma \ref{lem:multiplicativity} that
		$$X_{[\ddim M]_+}X_{[\ddim M]_-} = X_{[\ddim M]_++[\ddim M]_-} = X_{\ddim M} = X_{\den(c)}.$$
		This proves the first assertion. The second and third assertions are direct consequences of the first.
		
		We now prove the fourth assertion. Assume that $Q$ is a Dynkin quiver and fix some element $\textbf d \in \Z^{Q_0}$. Then $X_{\textbf d}=X_{[\textbf d]_+}X_{P_{\textbf d}[1]}$. According to Kac's canonical decomposition there exists $\textbf e_1, \ldots, \textbf e_n \in \Phi^{\Sc}(Q)$ such that $[\textbf d]_+=\textbf e_1 \oplus \cdots \oplus \textbf e_n$. In particular, a representation $M \in U_{[\textbf d]_+} \cap \mathfrak M_{[\textbf d]_+}$ can be written $M=M_1 \oplus \cdots \oplus M_n$ where each of the $M_i$ is indecomposable in $\mathfrak M_{\textbf e_i}$ and such that $\Ext^1_{H}(M_i,M_j)=0$ if $i \neq j$. Since $Q$ is Dynkin, each of the $M_i$ is rigid. It follows that $M$ is a rigid module. Moreover, for every $i \in Q_0$ such that $d_i<0$, we have $\Ext^1_{\CC_Q}(M,P_i[1])=0$ so that $M \oplus P_{\textbf d}[1]$ is rigid in $\CC_Q$. Thus, $X_{\textbf d}=X_{[\textbf d]_+}X_{P_{\textbf d}[1]}=X_{M \oplus P_{\textbf d}[1]}$ is the image of a rigid object and thus, it is a cluster monomial. 
		
		Conversely, fix $Q$ a representation-infinite quiver and assume that $\mathcal G(Q)=\mathcal M(Q)$. It follows that for every $\textbf d \in \N^{Q_0}$, there exists some rigid object $M$ in $\CC_Q$ such that $X_M=X_{\textbf d}$. According to Remark \ref{rmq:denXd}, $\textbf d=\den(X_{\textbf d})=\ddim M$. Since $M$ is rigid in $\CC_Q$ and $\ddim M \in \N^{Q_0}$, it follows that $M$ is a rigid $H$-module of dimension vector $\textbf d$. In particular, if $\textbf d$ is a positive imaginary root of $Q$, we get $\dim \End_{H}(M)-\dim \Ext^1_{H}(M,M)=\<M,M\>=\<\textbf d,\textbf d\> \leq 0$, so that $\dim \Ext^1_{H}(M,M)>0$ which is a contradiction. Thus, $\mathcal M(Q)$ is a proper subset of $\mathcal G(Q)$. 
	\end{proof}
\end{section}

\begin{section}{The affine case}
	In this section, we provide an explicit description of $\mathcal G(Q)$ when $Q$ is a quiver of affine type, that is, an acyclic quiver whose underlying diagram is an affine (also called euclidean or extended Dynkin) diagram of type $\Aaffine_n$ ($n \geq 1$), $\Daffine_n$ ($n \geq 4$) or $\Eaffine_n$ ($n=6,7,8$). An explicit list of these diagrams can for instance be found in \cite{ringel:1099}. We recall some well-known facts concerning representation theory of affine quivers. For more detailed results, we refer the reader to \cite{ringel:1099} or \cite{SS:volume2}. 

	There exists an unique element $\delta \in \Phi_{>0}(Q)$ such that $\Phi^{\im}_{>0}(Q)=\Z_{>0} \delta$. The root $\delta$ is called the \emph{positive minimal imaginary root} of $Q$. This root is \emph{sincere}, that is, $\delta_i>0$ for any $i \in Q_0$. Moreover, there exists an \emph{extending vertex} $e \in Q_0$, that is, a vertex $e$ such that $\delta_e=1$ and such that the quiver obtained from $Q$ by removing $e$ is of Dynkin type.

	The Auslander-Reiten quiver $\Gamma(H\modg)$ of $H$-mod contains infinitely many connected components. There exists a connected component containing all the projective (resp. injective) modules, called \emph{preprojective} (resp. \emph{preinjective}) component of $\Gamma(H\modg)$ and denoted by $\mathcal P$ (resp. $\mathcal I$). The other components are called \emph{regular}. A $H$-module $M$ is called \emph{preprojective} (resp. \emph{preinjective}, \emph{regular}) if each indecomposable direct summand of $M$ belong to a preprojective (resp. preinjective, regular) component. It is known that if $P$ is a preprojective module, $I$ a preinjective module and $R$ a regular module, then
	$$\Hom_{H}(R,P)=0, \, \Hom_{H}(I,P)=0, \, \Hom_{H}(I,R)=0,$$
	$$\Ext^1_{H}(P,R)=0, \, \Ext^1_{H}(P,I)=0 \textrm{ and } \Ext^1_{H}(R,I)=0.$$

	The regular components in the Auslander-Reiten quiver $\Gamma(H\modg)$ of $H$-mod form a $\P^1(\k)$-family of tubes and we denote by $\lambda \mapsto \mathcal T_\lambda$ this parametrization. For simplicity, we set $\P^1=\P^1(\k)$. We will recall an explicit description of this parametrization in the proof of Lemma \ref{lem:XMlambda}. The tubes of rank one are called \emph{homogeneous}, the tubes of rank $p>1$ are called \emph{exceptional}. At most three tubes are exceptional. We denote by $\P^{\mathcal H}$ the set of parameters of homogeneous tubes. For any $\lambda \in \P^{\mathcal H}$, we denote by $M_\lambda$ the unique quasi-simple module in the tube $\mathcal T_{\lambda}$. 

	There are neither morphisms nor extensions between the tubes in the sense that if $\lambda \neq \mu \in \P^1$ and $M,N$ are indecomposable modules respectively in $\mathcal T_\lambda$ and $\mathcal T_\mu$, then $\Hom_{H}(M,N)=0$ and $\Ext^1_{H}(M,N)=0$.

	For any quasi-simple module $R$ in a tube $\mathcal T$ and any integer $n \geq 1$, we denote by $R^{(n)}$ the unique indecomposable $H$-module with quasi-socle $R$ and quasi-length $n$. With these notations, for any $n \geq 1$ and any $\lambda \in \P^{\mathcal H}$, we have $\ddim M_\lambda^{(n)}=n \delta$. 

	An indecomposable regular module $M$ is rigid if and only if it is contained in an exceptional tube and $\ddim M \lneqq \delta$. An indecomposable regular module $M$ is a Schur module if and only if $\ddim M \leq \delta$.

	It is convenient to introduce the so-called \emph{defect form} on $\Z^{Q_0}$ given by
	$$\d_? : \left\{\begin{array}{rcl}
		\Z^{Q_0} & \fl & \Z \\
		\textbf e & \mapsto & \d_{\textbf e}=\<\delta, \textbf e\>
	\end{array}\right.$$
	By definition, the defect $\d_M$ of a $H$-module $M$ is the defect $\d_{\ddim M}$ of its dimension vector. It is well-known that an indecomposable $H$-module $M$ is preprojective (resp. preinjective, regular) if and only if $\d_M < 0$ (resp. $>0$, $=0$).
	
	\begin{subsection}{Generic variables associated to positive real Schur roots}
		In the case of real Schur roots, we prove that the corresponding generic variables are cluster variables.
		\begin{monlem}\label{lem:realSchur}
			Let $Q$ be an affine quiver and $\textbf d \in \Phi^{\Sc}(Q)$. Then $X_{\textbf d}=X_{M}$ where $M$ is the unique (up to isomorphism) indecomposable (and necessarily rigid) module of dimension $\textbf d$.
		\end{monlem}
		\begin{proof}
			As $\textbf d$ is a real root, Kac's theorem ensures that there exists an unique indecomposable representation $M$ of dimension $\textbf d$. Moreover, as $\textbf d$ is a Schur root, this representation has a trivial endomorphism ring. Now $1=\<\textbf d, \textbf d\>=\<M,M\>=\dim \End_{H}(M)-\dim \Ext^1_{H}(M,M)$ so $\dim \Ext^1_{H}(M,M)=0$. It follows from Proposition \ref{prop:clustermonomials} that $X_M=X_{\ddim M}=X_{\textbf d}$. 
		\end{proof} 
	\end{subsection}

	\begin{subsection}{Generic variables associated to positive imaginary roots}\label{ssection:genericim}
		For any $n \geq 0$, we denote by $S_n$ the $n$-th \emph{normalized Chebyshev polynomial of the second kind} (see Section \ref{section:basesKronecker} for a definition).
		
		\begin{monlem}[\cite{CZ}]\label{lem:Chebyshev}
			Fix $Q$ an affine quiver, $n \geq 1$ and $\lambda \in \P^{\mathcal H}$. Then,
			$$X_{M_\lambda^{(n)}}=S_n(X_{M_\lambda}).$$
		\end{monlem}

		We now prove that if $M$ is an indecomposable module in an homogeneous tube, then $X_M$ only depends on the quasi-length of $M$ and not on its quasi-socle.
		\begin{monlem}\label{lem:XMlambda}
			Let $Q$ be an affine quiver. Then for any $\lambda,\mu \in \P^{\mathcal H}$ and any $n \geq 1$, we have 
			$$X_{M_\lambda^{(n)}}=X_{M_\mu^{(n)}}.$$
		\end{monlem}
		\begin{proof}
			By Lemma \ref{lem:Chebyshev}, it is enough to prove it for $n=1$. For this, we recall Crawley-Boevey's construction of regular representations of dimension $\delta$ by one-point extensions \cite{CB:lectures}. Let $e$ be an extending vertex of $Q$. We set $P=P_e$ and $\textbf p=\ddim P$. Since $e$ is an extending vertex, then $\d_P=-1$. Let $L$ be the unique indecomposable representation of dimension vector $\delta +\textbf p$. Then $L$ has also defect $-1$. Then $\Hom_{H}(P,L) \simeq \k^2$ and for any morphism $0 \neq \lambda \in \Hom_{H}(P,L)$, $\coker \lambda$ is an indecomposable regular module of dimension vector $\delta$. Moreover, $\coker \lambda \simeq \coker \lambda'$ if and only if $\lambda = \alpha \lambda'$ for some $0 \neq \alpha \in \k$. Then $\lambda \mapsto \coker \lambda$ induces a bijection from $\P\Hom_{H}(P,L)$ to the set of all tubes of $\Gamma(H\modg)$ by sending $\lambda \in \P\Hom_{H}(P,L)$ to an indecomposable regular module of dimension $\delta$. We thus identify $\P\Hom_{H}(P,L)$ and $\P^1(\k)$ and we say that $\lambda \in \P^{\mathcal H}$ to signify that $\coker \lambda$ is a quasi-simple in the homogeneous tube $\mathcal T_{\lambda}$. Note that in this case $M_\lambda \simeq \coker \lambda$.
			
			Fix $0 \neq \lambda \in \Hom_{H}(P,L)$, then there are isomorphisms of $\k$-vectors spaces
			\begin{align*}
				\Ext^1_{\mathcal C_Q}(P,\coker \lambda) 
					& \simeq \Ext^1_{H}(P,\coker \lambda) \oplus \Ext^1_{H}(\coker \lambda,P) \\
					& \simeq \Ext^1_{H}(\coker \lambda,P) \\
					& \simeq \Hom_{H}(P,\coker \lambda) \\
					& \simeq \k
			\end{align*}
			where the last equality follows from the fact that $\dim \Hom_{H}(P,\coker \lambda)= (\ddim \coker \lambda)_e = \delta_e = 1$ since $e$ is an extending vertex.
			
			It follows from \cite[Theorem 2]{CK2} formula that 
			$$X_PX_{\coker \lambda}=X_L + X_B$$
			where $B \simeq \ker f \oplus \coker f[-1]$ for any non-zero $f \in \Hom_{H}(P,\tau \coker \lambda)$.
			
			Fix $\lambda \in \P^{\mathcal H}$. Assume first that $0 \neq f$ is not surjective in $\Hom_{H}(P,\coker \lambda)$, then $\im f$ is preprojective and we have a short exact sequence
			$$0 \fl \ker f \fl P \xrightarrow{f} \im f \fl 0$$
			and thus $\d_{\ker f}+\d_{\im f}=\d_P=-1$. Since $\ker f$ and $\im f$ are preinjective, we necessarily have $\d_{\ker f}=0$ and $\ker f=0$. 
			It follows that $B \simeq \tau^{-1}\coker f$. But $\d_{\coker f} =\d_{(\tau \coker \lambda)/P}=-\d_P=1$ and $\tau \coker \lambda \simeq \coker \lambda$ is quasi-simple in an homogeneous tube. Thus, $\coker f$ is preinjective and by defect, it is indecomposable. Moreover, the dimension vector of $\coker f$ is $\ddim \coker f=\delta-\textbf p$. Thus, for every $\lambda \in \P^{\mathcal H}$ and every non-surjective map $0 \neq f \in \Hom_{H}(P,\tau \coker \lambda)$, $\coker f$ is the unique indecomposable representation $V$ of dimension $\delta-\textbf p$. If there exists some non-surjective $f \in \Hom_{H}(P,\tau \coker \lambda)$, then $f$ is injective, so $\ddim P \lneqq \delta$ and thus every map $g \in \Hom_{H}(P,\tau \coker \lambda)$ is non-surjective for any $\lambda \in \P^{\mathcal H}$. Thus,
			$$X_PX_{\coker \lambda}=X_L+X_V$$
			for any $\lambda \in \P^{\mathcal H}$. In particular, $X_{M_\lambda}$ does not depend on the parameter $\lambda \in \P^{\mathcal H}$.
			
			Assume now that $0 \neq f \in \Hom_{H}(P,\tau \coker \lambda)$ is surjective. Then $\ker f$ is a preprojective module of dimension $\textbf p-\delta$. Thus, $\ker f$ has defect $-1$, it is thus necessarily indecomposable. As there exists an unique indecomposable representation $U$ of dimension $\textbf p-\delta$, it follows that $\ker f \simeq U$ for every $\lambda \in \P^{\mathcal H}$ and every surjective $f \in \Hom_{H}(P,\tau \coker \lambda)$. Then, it follows from the above discussion that if one of the $f \in \Hom_{H}(P,\tau \coker \lambda)$ is surjective for some $\lambda \in \P^{\mathcal H}$, then $f\in \Hom_{H}(P,\tau \coker \lambda)$ is surjective for any $\lambda \in \P^{\mathcal H}$. In this case, we get
			$$X_PX_{\coker \lambda}=X_L+X_U$$
			for every $\lambda \in \P^{\mathcal H}$. In particular, $X_{M_\lambda}$ does not depend on the parameter $\lambda \in \P^{\mathcal H}$. 
		\end{proof}

		If $\mathcal T$ is an exceptional tube of rank $p>1$ and $E$ is a quasi-simple module in $\mathcal T$, we denote by $M_E=E^{(p)}$ the unique indecomposable module in $\mathcal T$ with quasi-socle $E$ and quasi-length $p$. Equivalently, $M_E$ is the unique indecomposable module of dimension $\delta$ containing $E$ as a submodule.

		\begin{monlem}\label{lem:imaginary}
			For any $n \geq 1$ and any $\lambda \in \P^{\mathcal H}$, 
			$$X_{n \delta}=X_{M_\lambda}^n.$$
		\end{monlem}
		\begin{proof}
			We first prove it for $n=1$. It is known that the canonical decomposition of $\delta$ is $\delta$ itself. It follows that $X_\delta=X_M$ for some indecomposable module $M$ of dimension vector $\delta$. Now, we know that the indecomposable modules of dimension vector $\delta$ are either the $M_\lambda$ for $\lambda \in \P^{\mathcal H}$, or the $M_E$ for $E$ quasi-simple in an exceptional tube. Fix now a quasi-simple $E$ in an exceptional tube, we claim that $\Gr_{\ddim E}(M)=\emptyset$ for any indecomposable module $M$ of dimension vector $\delta$ not isomorphic to $M_E$. Indeed, fix $\lambda \in \P^{\mathcal H}$ and $U \subset M_\lambda$ a submodule such that $\ddim U=\ddim E$. As $M_\lambda$ is quasi-simple, $U$ has to be preprojective and thus $\d_U <0$ but $\d_U=\d_E=0$, which is a contradiction and thus $\Gr_{\ddim E}(M_\lambda)=\emptyset$. Fix now $F$ another quasi-simple and assume that $U \subset M_F$ is a submodule such that $\ddim U=\ddim E$. It follows that $U$ decomposes into $U=U_P \oplus U_R$ where $U_P$ is preprojective and $U_R$ is regular. As $\d_U=\d_E=0$, we have $U_P=0$ and thus $U_R$ is regular. Now $\ddim U_R=\ddim E$ but $U_R$ is a regular submodule $M_F$, by uniseriality of the regular components, $U_R$ has to be indecomposable. As $\ddim E$ is a real root, it follows that $E \simeq U_R \subset M_F$ and thus $F=E$. As the $GL(\delta)$-orbit of $M_E$ is not open (since its codimension in $\rep(Q,\delta)$ is equal to $\dim \Ext^1_{H}(M_E,M_E)=1$), it follows that the value of $\chi(\Gr_{\ddim E}(-))$ has to be zero on $U_{\delta,\ddim E}$ so that $U_\delta \cap \mathcal O_{M_E} = \emptyset$ for any quasi-simple $E$ in an exceptional tube. It follows that 
			$$U_{\delta} \cap \mathfrak M_{\delta} \subset \bigsqcup_{\lambda \in \P^{\mathcal H}} \mathcal O_{M_\lambda}.$$ Thus, $X_\delta=X_{M_\lambda}$ for some $\lambda \in \P^{\mathcal H}$. Lemma \ref{lem:XMlambda} implies that $X_\delta=X_{M_\lambda}$ for any $\lambda \in \P^{\mathcal H}$.
			
			Now if $n >1$, the canonical decomposition of $n\delta$ is known to be $\delta \oplus \cdots \oplus \delta$. Thus, Proposition \ref{prop:dcpcanonique} implies that $X_{n\delta}=X_\delta^n=X_{M_\lambda}^n$.
		\end{proof}

		\begin{rmq}
			One should take care of the fact that the generic variable of dimension $n \delta$ for $n \geq 2$ is not the character associated to an indecomposable module of dimension $n \delta$ in an homogeneous tube. Indeed, for every $\lambda \in \P^{\mathcal H}$, we have $X_{n\delta}=X_{\delta}^n \neq S_n(X_\delta)=X_{M_\lambda^{(n)}}$ if $n \geq 2$. The indecomposable modules of dimension $n \delta$ do not appear in the context of generic variables, nevertheless they arise in Caldero-Zelevinsky construction (see Section \ref{section:basesKronecker} or \cite{CZ} for details).
		\end{rmq}
	\end{subsection}

	\begin{subsection}{Generic variables associated to positive real non-Schur roots}
		We now compute $X_{\textbf d}$ when $\textbf d$ is a real root which is not a Schur root. If $\textbf d$ is such a root, then according to Kac's theorem, there exists an unique indecomposable module $M$ of dimension vector $\textbf d$. If $M$ is preprojective or preinjective, it is known that $\End_{H}(M) \simeq \k$ and thus $\textbf d$ is a real Schur root. It follows that $M$ has to be a regular module. Moreover, we have $\delta \lneqq \textbf d$. Fix $\textbf d$ such a root, it is proved in \cite{Kac:infroot2} that $\textbf d = \delta \oplus \cdots \oplus \delta \oplus \textbf d_0$ where $\textbf d_0$ is the root of smallest height in $\textbf d+ \Z \delta \cap \Phi_{>0}(Q)$. In particular, $\textbf d_0$ is a real Schur root. The following lemma gives an explicit description of $X_{\textbf d}$ in this case.
		
		\begin{monlem}\label{lem:realnonSchur}
			Let $Q$ be an affine quiver, $\textbf d$ be a real root which is not a Schur root and write $\textbf d=\delta^{\oplus n} \oplus \textbf d_0$ its canonical decomposition. Then, 
			$$X_{\textbf d}=X_{M_\lambda}^nX_{M_0}$$
			where $\lambda$ is any element in $\P^{\mathcal H}$ and where $M_0$ is the unique (up to isomorphism) indecomposable module of dimension vector $\textbf d_0$. Moreover, $M_0$ is rigid.
		\end{monlem}
		\begin{proof}
			According to Lemma \ref{lem:realSchur}, $X_{\textbf d_0}=X_{M_0}$ where $M_0$ is the unique indecomposable module in $\rep(Q,\textbf d_0)$. According to Lemma \ref{lem:imaginary}, $X_{n \delta}=X_{\delta}^n=X_{M_\lambda}^n$ for any $\lambda \in \P^{\mathcal H}$. Proposition \ref{prop:dcpcanonique} thus implies that $X_{\textbf d}=X_{n\delta}X_{\textbf d_0}=X_{M_\lambda}^nX_{M_0}$.
		\end{proof}
		
		We can now give a complete description of the generic variables~:
		\begin{maprop}\label{prop:explicitbase}
			Let $Q$ be an affine quiver. Then,
			$$\mathcal G(Q)=\mathcal M(Q) \sqcup \ens{X_{\delta}^nX_E \ : n \geq 1, \ E \textrm{ is a regular rigid $H$-module}}.$$
		\end{maprop}
		\begin{proof}
			Fix $\textbf d \in \Z^{Q_0}$ and write $[\textbf d]_+=\delta^k \oplus \bigoplus_{i=1}^n e_i$ the canonical decomposition of $[\textbf d]_+$ with $n,k \geq 0$.
		
			If $k \neq 0$, then as $\delta$ is sincere, we have $\textbf d \in \N^{Q_0}$. According to Proposition \ref{prop:Kacdcp} and Lemma \ref{lem:imaginary}, there exists some $M$ in $\mathfrak M_{\textbf d} \cap U_{\textbf d}$ such that
			$$M=\bigoplus_{j=1}^k M_{\lambda_j} \oplus \bigoplus_{i=1}^n M_i$$
			where the $\lambda_j$ are pairwise distinct elements of $\P^{\mathcal H}$ and the $M_i$ are indecomposable rigid modules such that $\End_{H}(M_i) \simeq \k$ and $\Ext^1_{H}(M_i,M_l)=0$ if $i \neq l$.
			
			If one of the $M_i$ is preprojective, then there is some vertex $v \in Q_0$ and some integer $s \geq 0$ such that $M_i\simeq \tau^{-s}P_v$ and then $$\Ext^1_{\CC_Q}(M_i, M_1)=\Hom_{\CC_Q}(P_v, M_1) =\dim M_1(v)=\delta_v \geq 1$$ which is a contradiction. Similarly, none of the $M_i$ can be preinjective. It follows that each $M_i$ is regular and thus the $M_i$ are indecomposable regular modules in exceptional tubes such that $\ddim M_i \lneqq \delta$. In particular, each $M_i$ is a rigid $H$-module. Since $\Ext^1_{H}(M_k,M_l)=0$ for $k \neq l$, it follows that $\bigoplus_{i=1}^n M_i$ is rigid.
			
			If $k=0$, then $M \in \mathfrak M_{\textbf d} \cap U_{\textbf d}$ is rigid and thus $M \oplus P_{\textbf d}[1]$ is a rigid object in $\CC_Q$ so that $X_{\textbf d}$ is a cluster monomial of $\mathcal A(Q)$.
			
			This proves the inclusion 
			$$\mathcal G(Q) \subset \mathcal M(Q) \sqcup \ens{X_{\delta}^nX_E \ : n \geq 1, \ E \textrm{ is a regular rigid $H$-module}}.$$
			
			We now prove the reverse inclusion. According to Proposition \ref{prop:clustermonomials} it suffices to prove that 
			$$\ens{X_{\delta}^nX_E \ :n \geq 1 \ E \textrm{ is a regular rigid $H$-module}} \subset \mathcal G(Q).$$ 
			Fix some regular rigid module $E$ and some $n \geq 1$. Then, $X_{\ddim E}=X_E$ by Proposition \ref{prop:clustermonomials}. Fix now $\lambda_1, \ldots, \lambda_n$ pairwise distinct elements in $\P^{\mathcal H}$. According to Lemma \ref{lem:imaginary}, we have $X_{\delta}^n=X_{M_{\lambda_1}}\cdots X_{M_{\lambda_n}}$. Since there are no extensions between the tubes, we have
			$$\Ext^1_{\mathcal C_Q}(E, M_{\lambda_1} \oplus \cdots \oplus M_{\lambda_n})=0$$
			so that $\Ext^1_{\mathcal C_Q}(\ddim E, n \delta)$ vanishes generally. By Lemma \ref{lem:multiplicativity}, we thus have 
			$$X_{\ddim E + n \delta}=X_{\ddim E}X_{n \delta}$$
			and this proves that
			$$\mathcal G(Q)=\mathcal M(Q) \cup \ens{X_{\delta}^nX_E \ : n \geq 1, \ E \textrm{ is a regular rigid $H$-module}}.$$
			
			It remains to notice that the union is disjoint. Let $E$ be a regular rigid module and $n \geq 1$. Assume that there exists some rigid object $M$ in $\CC_Q$ such that $X_M=X_\delta^n X_E$. According to \cite[Theorem 3]{CK2}, it follows that $\ddim M=n \delta + \ddim E$. Since $M$ is rigid and $\delta$ is sincere, it follows that $M$ is a rigid $H$-module which is a contradiction since there exists no rigid module of dimension vector $n \delta + \ddim E$. 
		\end{proof}

		\begin{theorem}\label{theorem:imCC}
			Let $Q$ be an affine quiver. Then 
			$$\mathcal A(Q)=\im(X_?).$$
		\end{theorem}
		\begin{proof}
			Since cluster variables are characters of indecomposable rigid objects in $\CC_Q$, we know that
			$$\mathcal A(Q) \subset \Z[X_M |M \in \Ob(\CC_Q)]=\im(X_?).$$
			It is thus sufficient to prove the reverse inclusion.
			
			Since $X_MX_N=X_{M \oplus N}$ for any two objects $M,N$ in $\CC_Q$, it suffices to prove that $X_M \in \mathcal A(Q)$ for any indecomposable object $M$ in $\CC_Q$. If $M$ is not a regular $H$-module, $M$ is rigid and thus $X_M$ is a cluster variable. If $M$ is a regular $H$-module, it is contained in a tube $\mathcal T$. We can thus write $M \simeq E^{(n)}$ for some quasi-simple module $E$ in $\mathcal T$ and $n \geq 1$. It follows from \cite{Dupont:stabletubes} that 
			$$X_M=P_{n}(X_{E}, \ldots, X_{\tau^{-n+1}E})$$
			where $P_{n}$ is the $n$-th generalized Chebyshev polynomial of infinite rank introduced in \cite{Dupont:stabletubes}. If $\mathcal T$ is exceptional, then $\tau^k E$ is rigid for any $k \in \Z$. Thus, $X_{\tau^k E}$ is a cluster variable for any $k \in \Z$ and $X_M$ is a polynomial in cluster variables and thus is an element of $\mathcal A(Q)$. If $\mathcal T$ is homogeneous, $E \simeq M_\lambda$ for some $\lambda \in \P^{\mathcal H}$ and it is thus enough to prove that $X_{M_\lambda} \in \mathcal A(Q)$. 
			
			Fix thus $\lambda \in \P^{\mathcal H}$. With the notations of the proof of Lemma \ref{lem:XMlambda}, we have
			$$\Ext^1_{\CC_Q}(M_\lambda, P) \simeq \Ext^1_{H}(M_\lambda, P) \simeq \k$$
			and the corresponding triangles are
			$$P \fl L \fl M_\lambda \fl P[1],$$
			$$M_\lambda \fl B \fl P \xrightarrow{f} M_\lambda[1] \simeq M_\lambda$$
			where $B \simeq \ker f \oplus \coker f[-1]$. Thus we get $X_PX_{M_\lambda}=X_L+X_B$.
			
			Fix now any cluster $\textbf c$ in $\mathcal A(Q)$, denote by $T$ the corresponding cluster-tilting object in $\CC_Q$. We denote by $Q_T$ the ordinary quiver of the cluster-tilted algebra $\End_{\CC_Q}(T)^{\rm{op}}$ and by $X^T_?: \Ob(\CC_Q) \fl \Z[\textbf c^{\pm 1}]$ the cluster character in the sense of Palu introduced in \cite{Palu}. Denote by $\phi: \mathcal A(Q,\textbf u) \fl \mathcal A(Q_T,\textbf c)$ the canonical isomorphism of cluster algebras sending each $u_i$ to its Laurent expansion in the cluster $\textbf c$. It follows from \cite{Palu} that $\phi(X_M)=X^T_M$ for any rigid object $M$ in $\CC_Q$. In particular $\phi(X_L)=X^T_L$, $\phi(X_P)=X^T_P$ and since $\ker f$ is preprojective and $\coker f[-1]$ is preinjective, it follows that $\phi(X_B)=X^T_B$. We thus get
			\begin{align*}
				X^T_P\phi(X_{M_\lambda})
					& = \phi(X_P)\phi(X_{M_\lambda})\\
			 		& = \phi(X_PX_{M_\lambda})\\
					& = \phi(X_L+X_B)\\
			 		& = \phi(X_L)+\phi(X_B)\\
			 		& = X^T_L+X^T_B.
			\end{align*}
			Now the one-dimensional multiplication formula for cluster characters in \cite{Palu} implies that $X^T_PX^T_{M_\lambda}=X^T_L+X^T_B$. Since $\Z[\textbf c^{\pm 1}]$ is a domain, it follows that
			$$\phi(X_{M_\lambda})=X^T_{M_\lambda}.$$
			
			Thus, for any cluster $\textbf c$, we have $X_{M_\lambda}\in \Z[\textbf c^{\pm 1}]$ and thus 
			$$X_{M_\lambda} \in \bigcap_{\textbf c} \Z[\textbf c^{\pm 1}]$$ 
			where the intersection runs over all the clusters $\mathcal A(Q)$.
			
			In particular, if we denote by $\overline{\mathcal A(Q)}$ the upper cluster algebra associated to $Q$ (see \cite{cluster3}), we get $X_{M_\lambda} \in \overline{\mathcal A(Q)}$. Now since $Q$ is acyclic, the upper cluster algebra $\overline{\mathcal A(Q)}$ coincides with the cluster algebra $\mathcal A(Q)$ (see e.g. \cite{cdm03}) and thus, $X_{M_\lambda} \in \mathcal A(Q)$. This finishes the proof.  
		\end{proof}

		As an immediate corollary, we get~:
		\begin{corol}
		 	Let $Q$ be an affine quiver, then $\mathcal G(Q) \subset \mathcal A(Q)$.
		\end{corol}
	\end{subsection}
\end{section}

\begin{section}{The Kronecker case}\label{section:basesKronecker}
	We now study in details the generic variables in the cluster algebra of affine type associated to the Kronecker quiver 
	$$\xymatrix{ K : 1 & \ar@<-2pt>[l]\ar@<+2pt>[l] 2}.$$
	We denote by $\delta=(1,1)$ the minimal positive imaginary root of $K$ and we denote by $\mathcal A(K)$ the coefficient-free cluster algebra with initial seed $(K, \textbf u)$ with $\textbf u=(u_1,u_2)$. We compare the set $\mathcal G(K)$ of generic variables in $\mathcal A(K)$ with $\Z$-linear bases in $\mathcal A(K)$ constructed by Sherman-Zelevinsky \cite{shermanz} and Caldero-Zelevinsky \cite{CZ}. This proves in particular that $\mathcal G(K)$ is a $\Z$-basis in $\mathcal A(K)$.

	\begin{subsection}{Normalized Chebyshev polynomials}\label{subsection:CZChebyshev}
		In this subsection, we recall briefly some results concerning the normalized Chebyschev polynomials of first and second kinds. For further details concerning these polynomials, especially in the context of cluster algebras, one can refer to \cite{Dupont:qChebyshev}.
		
		We recall that the \emph{normalized Chebyshev polynomials of the second kind} are the polynomials defined inductively by~:
		$$S_0(x)=1, \, S_1(x)=x \textrm{ and } S_{n+1}(x)=xS_n(x)-S_{n-1}(x)  \textrm{ for any }n \geq 1.$$
		It is known that for every $n \geq 0$, $S_n$ is the monic polynomial of degree $n$ characterized by 
		$$S_n(t+t^{-1})=\sum_{k=0}^n t^{n-2k}$$

		The \emph{normalized Chebyschev polynomials of the first kind} are the polynomials defined inductively by :
		$$F_0(x)=2, \, F_1(x)=x \textrm{ and } F_{n+1}(x)=xF_n(x)-F_{n-1}(x)  \textrm{ for any }n \geq 1.$$
		It is known that for every $n \geq 0$, $F_n$ is the monic polynomial of degree $n$ characterized by 
		$$F_n(t+t^{-1})=t^n+t^{-n}.$$
		
		Chebyshev polynomials of the first and second kinds are related by
		$$S_n(x)=\sum_{k=1}^{\left[\frac n2\right]+1}F_{n-2k}(x).$$
		In particular each second kind Chebyshev polynomial is a positive linear combination of first kind Chebyshev polynomials.
	\end{subsection}

	\begin{subsection}{Sherman-Zelevinsky basis}
		We recall some results and notations of \cite{shermanz}. 
		\begin{defi}
			An element $y \in \mathcal A(K)$ is called positive if, for every cluster $\textbf x=(x,x')$ of $\mathcal A(K)$, the coefficients in the expansion of $y$ as a Laurent polynomial in $x$ and $x'$ are positive.
		\end{defi}
		
		\begin{theorem}[\cite{shermanz}]\label{theorem:canonicalbasis}
			There exists an unique $\Z$-basis $\mathcal B(K)$ of $\mathcal A(K)$ such that the semi-ring of positive elements in $\mathcal A(K)$ consists precisely of $\Z_{\geq 0}$-linear combinations of elements of $\mathcal B(K)$. Moreover, this basis is given by
			$$\mathcal B(K)=\mathcal M(K) \sqcup \ens{F_n(z) |n \geq 1}$$
			where 
			$$z=\frac{1+u_1^2+u_2^2}{u_1u_2}.$$
		\end{theorem}
		It is proved in \cite{shermanz} that $\den(F_n(z))=n \delta$ for any $n \geq 1$ so that the basis $\mathcal B(K)$ is the disjoint union of the set of cluster monomials together with a family of variables whose denominator vectors are the positive imaginary roots of $Q$.
	\end{subsection}

	\begin{subsection}{Caldero-Zelevinsky basis}
		In \cite{CZ}, the authors computed another $\Z$-basis for the cluster algebra associated to the Kronecker quiver. This basis, constructed using the Caldero-Chapoton map, is given by
		$$\mathcal C(K)=\mathcal M(K) \sqcup \ens{X_{M_\lambda^{(n)}} |n \geq 1}$$
		where $M_\lambda$ is any quasi-simple regular module. Moreover, it was observed that $X_{M_\lambda}=z$.
		
		According to Lemma \ref{lem:imaginary}, we have $X_{\delta}=X_{M_\lambda}$ and Lemma \ref{lem:Chebyshev} implies that $X_{M_\lambda^{(n)}}=S_n(X_{M_\lambda})$ for any $n \geq 1$. Thus,
		$$\mathcal C(K)=\mathcal M(K) \sqcup \ens{S_n(z) |n \geq 1}$$

		According to \cite[Theorem 3]{CK2}, we have $\den(S_n(z))=\den(X_{M_\lambda^{(n)}})=\ddim(M_\lambda^{(n)})=n\delta$, so that the basis $\mathcal C(K)$ is also the disjoint union of the set of cluster monomials and of a set of variables whose denominator vectors correspond to the positive imaginary roots of $Q$.
	\end{subsection}

	\begin{subsection}{Generic variables for the Kronecker quiver}
		We now describe the set $\mathcal G(K)$ of generic variables in $\mathcal A(K)$. It follows from Proposition \ref{prop:explicitbase} that
		$$\mathcal G(K)=\mathcal M(K) \sqcup  \ens{z^n |n \geq 1}.$$
		
		For any $n \geq 1$, the denominator vector $\den(z^n)$ of $z^n$ is $n \den(z)=n\delta$. Thus, as for Caldero-Zelevinsky and Sherman-Zelevinsky bases, the set of generic variables is the disjoint union of the set of cluster monomials and a family of variables whose denominator vector are the positive imaginary roots of $K$.

		Using Theorem \ref{theorem:canonicalbasis} and the fact that the $F_n$ (or the $S_n)$ are monic, we get~:
		\begin{corol}
			The set $\mathcal G(K)$ of generic variables is a $\Z$-basis of the cluster algebra $\mathcal A(K)$.
		\end{corol}
	\end{subsection}

	\begin{subsection}{Base change between $\mathcal B(K)$ and $\mathcal G(K)$}
		\begin{defi}
			Let $\textbf a=\ens{a_n, n \geq 0}$ and $\textbf b=\ens{b_n, n \geq 0}$ be two bases of the $\Z$-module $\mathcal A(K)$. We say that there is a \emph{locally unipotent base change} from $\textbf a$ to $\textbf b$ if for every $n \in \Z$, the $\Z$-modules spanned by $\ens{a_k, 0 \leq k \leq n}$ and $\ens{b_k, 0 \leq k \leq n}$ coincide and if the base change matrix $P$ from $(a_k, 0 \leq k \leq n)$ to $(b_k, 0 \leq k \leq n)$ is unipotent in $M_n(\Z)$. If moreover $P$ has positive entries, then the base change is called \emph{positive}.
		\end{defi}
		
		\begin{maprop}\label{prop:BtoB'}
			There is a positive locally unipotent base change from $\mathcal B(K)$ to $\mathcal G(K)$.
		\end{maprop}
		\begin{proof}
			As 
			$$\mathcal B(K)=\ens{\textrm{cluster monomials}} \sqcup \ens{F_n(z)|n \in \N}$$
			and
			$$\mathcal G(K)=\ens{\textrm{cluster monomials}} \sqcup \ens{z^n|n \in \N},$$
			it suffices to prove that there is a positive locally unipotent base change from $\ens{F_n(z), n \in \N}$ to $\ens{z^n, n \in \N}$. It is equivalent to prove that every $z^n$ can be written as a positive $\Z$-linear combination of the $F_k(z)$ for $0 \leq k \leq n$ where the coefficient of $F_n(z)$ is 1.
			
			Each $F_k(z)$ being a monic polynomial of degree $k$, it follows that $z^n$ can be written as a $\Z$-linear combination of the $F_k(z)$ for $0 \leq k \leq n$ and the coefficient of $F_n(z)$ is 1. Thus, there is a locally unipotent base change from $\mathcal B(K)$ to $\mathcal G(K)$.
			
			According to \cite{shermanz}, $z=X_{M_\lambda}$ is a positive element in $\mathcal A(K)$. Since positive elements form a semiring in $\mathcal A(K)$, each $X_{M_\lambda}^n$ is a positive element in $\mathcal A(K)$ and can thus be written as a positive $\Z$-linear combination of elements of $\mathcal B(K)$. The base change is thus positive and the proposition is proved. 
		\end{proof}
		
		\begin{monexmp}
			If we look at the base change matrix $P$ from the family $(z^n, 0 \leq n \leq 6) \subset \mathcal G(K)$ to the family $(F_n(z), 0 \leq n \leq 6)$ of the Sherman-Zelevinsky basis, we obtain
			$$P=\left[\begin{array}{ccccccc}
				1 & 0 & 2 & 0 & 6 & 0 & 20 \\
				0 & 1 & 0 & 3 & 0 & 10 & 0 \\
				0 & 0 & 1 & 0 & 4 & 0 & 15 \\
				0 & 0 & 0 & 1 & 0 & 5 & 0 \\
				0 & 0 & 0 & 0 & 1 & 0 & 6 \\
				0 & 0 & 0 & 0 & 0 & 1 & 0 \\
				0 & 0 & 0 & 0 & 0 & 0 & 1 
			\end{array}\right],\, P^{-1}=\left[\begin{array}{rrrrrrr}
				1 & 0 & -2 & 0 & 2 & 0 & -2 \\
				0 & 1 & 0 & -3 & 0 & 5 & 0 \\
				0 & 0 & 1 & 0 & -4 & 0 & 9 \\
				0 & 0 & 0 & 1 & 0 & -5 & 0 \\
				0 & 0 & 0 & 0 & 1 & 0 & -6 \\
				0 & 0 & 0 & 0 & 0 & 1 & 0 \\
				0 & 0 & 0 & 0 & 0 & 0 & 1 
			\end{array}\right].$$ 
		\end{monexmp}
	\end{subsection}

	\begin{subsection}{Base change between $\mathcal G(K)$ and $\mathcal C(K)$}
		For any $n \geq 0$, we write $F_n=F_n(z)$ and
		$$z^n=\sum_{i \leq n} \lambda_{i,n} F_i$$
		the expansion of $z^n$ in the $F_n$. It follows from Proposition \ref{prop:BtoB'} that that each $\lambda_{i,n}$ is positive.
		
		\begin{monlem}\label{lemcoeff}
			For any $n \geq 1$, we have :
			\begin{enumerate}
				\item $\lambda_{i,n}=0$ if $i \not \equiv n [2]$,
				\item $\lambda_{i,n} < \lambda_{i-2,n}$ for any $i \geq 2$ such that $i \equiv n[2]$.
			\end{enumerate} 
		\end{monlem}
		\begin{proof}
			We prove it by induction on $n$. If $n=1$, then $z^n=F_1$ and thus $\lambda_{1,1}=1$ and $\lambda_{i,1}=0$ for all $i \neq 1$, the above assertions are true. We now prove the induction step. We have 
			$$z^{n+1}=z.z^n=F_1.\left(\sum_{i \leq n} \lambda_{i,n} F_i \right)$$
			Now according to \cite[prop. 5.4 (1)]{shermanz}, 
			$$F_1F_i=\left\{ \begin{array}{rl}
				F_{i-1}+F_{i+1} & \textrm{ if } n >1~; \\
				2 + F_2 & \textrm{ if } i =1~;\\
				F_1 & \textrm{ if } i=0.
			\end{array}\right.$$
			It follows that
			\begin{align*}
				z^{n+1}
					&= \lambda_{0,n}F_1F_0 + \lambda_{1,n} F_1F_1+\sum_{2 \leq i \leq n}\lambda_{i,n}F_1F_i \\
					&= \lambda_{0,n}F_1 + \lambda_{1,n} (2+F_2)+\sum_{2 \leq i \leq n}\lambda_{i,n}(F_{i-1}+F_{i+1}) \\
					&= 2\lambda_{1,n}F_0 + (\lambda_{0,n}+\lambda_{2,n})F_1 + \sum_{i \geq 2}(\lambda_{i-1,n}+\lambda_{i+1,n}) F_i
			\end{align*}
			A direct check proves that the induction step is satisfied. 
		\end{proof}

		\begin{maprop}\label{prop:B''toB'}
			There is a positive locally unipotent base change from $\mathcal C(K)$ to $\mathcal G(K)$.
		\end{maprop}
		\begin{proof}
			As 
			$$\mathcal G(K)=\ens{\textrm{cluster monomials}}\sqcup\ens{z^n, n \geq 0}$$
			and
			$$\mathcal C(K)=\ens{\textrm{cluster monomials}}\sqcup\ens{S_n(z), n \geq 0},$$
			it suffices to prove that for any $n \geq 0$, the coefficients of the expansion of $z^n$ as a linear combination of the $S_n(z)$ are positive.
			
			We set $S_n=S_n(z)$ and we recall that $F_n=S_n-S_{n-2}$. We write
			$$z^n=\sum_{i \leq n} \lambda_{i,n} F_n$$
			the expansion of $z^n$ as a linear combination of the $F_n$. Then
			$$z^n = \sum_{i} \lambda_{i,n} (S_n-S_{n-2}) = \sum_{i} (\lambda_{i,n}-\lambda_{i+2,n}) S_n.$$
			According to Lemma \ref{lemcoeff}, the difference $(\lambda_{i,n}-\lambda_{i+2,n})$ is positive. Then, $z^n$ can be written as a positive linear combination of the $S_n$. 
		\end{proof}
		
		\begin{monexmp}
			The base change matrix $P$ from $(S_n(z), 0 \leq n \leq 6) \subset \mathcal C(K)$ to $(z^n, 0 \leq n \leq 6) \subset \mathcal G(K)$ is 
			$$P=\left[\begin{array}{ccccccc}
				1 & 0 & 1 & 0 & 2 & 0 & 5 \\
				0 & 1 & 0 & 2 & 0 & 5 & 0 \\
				0 & 0 & 1 & 0 & 3 & 0 & 9 \\
				0 & 0 & 0 & 1 & 0 & 4 & 0 \\
				0 & 0 & 0 & 0 & 1 & 0 & 5 \\
				0 & 0 & 0 & 0 & 0 & 1 & 0 \\
				0 & 0 & 0 & 0 & 0 & 0 & 1 \\
			\end{array}\right], \, P^{-1} = \left[\begin{array}{rrrrrrr}
				1 & 0 & -1 & 0 & 1 & 0 & -1 \\
				0 & 1 & 0 & -2 & 0 & 3 & 0 \\
				0 & 0 & 1 & 0 & -3 & 0 & 6 \\
				0 & 0 & 0 & 1 & 0 & -4 & 0 \\
				0 & 0 & 0 & 0 & 1 & 0 & -5 \\
				0 & 0 & 0 & 0 & 0 & 1 & 0 \\
				0 & 0 & 0 & 0 & 0 & 0 & 1 
			\end{array}\right].$$
		\end{monexmp}
	\end{subsection}	
\end{section}

\begin{section}{Conjecture : Generic bases in acyclic cluster algebras}\label{section:conjectures}
		As mentioned in the introduction, the first aim in introducing generic variables is to provide a general construction of $\Z$-linear bases in acyclic cluster algebras. We conjecture~:
		\begin{maconj}
			Let $Q$ be an acyclic quiver. Then $\mathcal G(Q)$ is a $\Z$-basis in $\mathcal A(Q)$.
		\end{maconj}
		Based on results presented in \cite{mathese,DXX:basesv3}, we will provide in a forthcoming article a positive answer to this conjecture when $Q$ is any quiver of affine type. Jan Schr\"oer announced at the Conference \emph{Homological and Geometric Methods in Algebra} in August 2009 that, together with Christof Geiss and Bernard Leclerc, they proved that $\mathcal G(Q)$ is a $\C$-linear basis of $\mathcal A(Q) \otimes_{\Z} \C$ for any acyclic quiver $Q$.
\end{section} 

\section*{Acknowledgements}
	This work is part of my PhD thesis. I would like to thank my supervisor Philippe Caldero for all his comments, corrections and advices. I would also like to thank Bernhard Keller, Robert Marsh, Idun Reiten, Kenji Iohara and Jan Schr\"oer.


\newcommand{\etalchar}[1]{$^{#1}$}

\end{document}